\documentclass[12pt,reqno]{amsart}
\makeatletter
\def\l@subsection{\@tocline{2}{0pt}{2.5pc}{5pc}{}}
\def\l@subsubsection{\@tocline{2}{0pt}{5pc}{7.5pc}{}}
\makeatother

\usepackage{frcursive}
\usepackage{yfonts}
\usepackage{graphicx}
\usepackage{empheq}
\usepackage{mathtools}
\usepackage{tikz}
\usepackage{amsbsy}
\usepackage{avant}
\usepackage{amsmath}
\usepackage{amsthm}
\usepackage{amssymb}
\usepackage{mathrsfs}
\usepackage{amsfonts}
\usepackage{newlfont}
\usepackage[sans]{dsfont}
\usepackage{dcolumn}
\usepackage{bm}
\usepackage{bbm}
\usepackage{stmaryrd}
\usepackage{bbm}
\usepackage{relsize}
\usepackage{euscript}
\usepackage{eufrak}
\usepackage{enumerate}
\usepackage[margin=0.78in]{geometry}
\numberwithin{equation}{section}
\newtheorem{thm}{Theorem}[section]

\newtheorem{cor}[thm]{Corollary}
\newtheorem{lem}[thm]{Lemma}
\newtheorem{prop}[thm]{Proposition}
\newtheorem{defn}[thm]{Definition}

\begin{document}
\allowdisplaybreaks{
\title[]{A CLUSTER EXPANSION PROOF THAT THE STOCHASTIC EXPONENTIAL OF A BROWNIAN MOTION IS A MARTINGALE}
\author{Steven D Miller}\email{stevendm@ed-alumni.net}
\address{Rytalix Capital}
\maketitle
\begin{abstract}
Let ${\psi}:\mathbb{R}^{+}\rightarrow\mathbb{R}^{+}$ be a smooth and continuous real function and $\psi\in\mathrm{L}^{2}(\mathbb{R}^{+})$. Let ${B}(t)$ be a standard Brownian motion defined with respect to a probability space $(\Omega,\mathscr{F},\bm{\mathsf{P}})$ and where $d{B}(t)={\xi}(t)dt$ and $t\in\mathbb{R}^{+}$. The process $\xi(t)$ is a Gaussian white noise with expectation $\bm{\mathsf{E}}~\xi(t)=0$ and with covariance $\bm{\mathsf{E}}~\xi(t)\xi(s)=\delta(t-s)$. The Dolean-Dades stochastic exponential ${Z}(t)$ is the solution to the linear stochastic differential equation describing a geometric Brownian motion such that $d{Z}(t)=\psi(t){Z}(t)d{B}(t)=\psi(t){Z}(t)\xi(t)dt$. Using a cluster expansion method, and the moment and cumulant generating functions for $\xi(t)$, it is shown that ${Z}(t)$ is a martingale. The original Novikov criteria for ${Z}(t)$ being a true martingale are reproduced and exactly satisfied, namely that
\begin{align}
\bm{\mathsf{E}}{Z}(t)=\bm{\mathsf{E}}\exp\left(\int_{o}^{t}{\psi}(u)d{B}(u)
-\frac{1}{2}\int_{0}^{t}|{\psi}(u)|^{2}du\right)=1\nonumber
\end{align}
provided that $\exp\big(\int_{0}^{t}|{\psi}(u)|^{2}du\big)<\infty$ for all $t>0$. However, $\bm{\mathsf{E}}\big[|{Z}(t)|^{p}\big] =\exp(\tfrac{1}{2}p(p-1)\phi(t))$, if $\phi(t)=\int_{0}^{t}|\psi(u)|^{2}du$ is monotone increasing and is a submartingale for all $p>1$.
\end{abstract}
\raggedbottom
\maketitle
\section{\textbf{Introduction}}
Let ${B}(t)$ be a standard Brownian motion with respect to a probability space $(\Omega,\mathscr{F},\bm{\mathsf{P}})$, and let
${\psi}:\mathbb{R}^{+}\rightarrow\mathbb{R}^{+}$ be smooth continuous real function \textbf{[1]}. A classical problem in stochastic analysis is to prove that the Dolean-Dades stochastic exponential
\begin{align}
&{Z}(t)=\exp\left(\int_{o}^{t}{\psi}(u)d{B}(u)-\frac{1}{2}\int_{0}^{t}|{\psi}(u)|^{2}du\right)\\&
{Z}(0)=0
\end{align}
is a true martingale \textbf{[2]}. The DDSE is the exact solution of the geometric Brownian motion \textbf{[1]}
\begin{align}
{d{Z}(t)}={\psi}(t){Z}(t)~\xi(t)dt\equiv{\psi}(t){Z}(t)~d{B}(t)
\end{align}
where $d{B}(t)=\xi(t)dt$ and $\xi(t)$ is a Gaussian white noise. The solution has the Ito stochastic integral representation
\begin{align}
{Z}(t)=\int_{0}^{t}\psi(u){Z}(u)d{B}(u)
\end{align}
The explicit solution (1.1) is obtained \textbf{[1]} via the Ito expansion of $\log {Z}(t)$. Proving that ${Z}(t)$ is a true martingale is actually a somewhat
difficult and subtle problem and is relevant to other theorems and results \textbf{[2-7]}. Ito integrals of the form (1.4) are not always martingales. The well-known necessary and sufficient conditions for ${Z}(t)$ to be a true martingale are due to Novikov \textbf{[2]} and also Kazamaki \textbf{[5]}. The Novikov criteria are
\begin{align}
\bm{\mathsf{E}}{Z}(t)=\bm{\mathsf{E}}
\exp\left(\int_{o}^{t}{\psi}(u)d{B}(u)-\frac{1}{2}\int_{0}^{t}|{\psi}(u)|^{2}du\right)=1
\end{align}
and the bound
\begin{align}
\exp\left(\frac{1}{2}\int_{0}^{t}|{\psi}(u)|^{2}du\right)<\infty,~~\forall t>0
\end{align}
or equivalently $\frac{1}{2}\|\psi\|_{L_{2}(\mathbb{R})}^{2}=\frac{1}{2}\int_{0}^{t}|{\psi}(u)|^{2}du<\infty$. The function ${\psi}$ can also be a random process ${\psi}(t,\omega)$ with respect to $\omega\in\Omega$ with Novikov criteria \textbf{[6]}
\begin{align}
\bm{\mathsf{E}}{Z}(t)=\bm{\mathsf{E}}
\exp\left(\int_{o}^{t}{\psi}(u,\omega)d{B}(u,\omega)-\frac{1}{2}\int_{0}^{t}|{\psi}(u,\omega)|^{2}du\right)=1
\end{align}
and
\begin{align}
\bm{\mathsf{E}}\left \lbrace\exp\left(\int_{0}^{t}|{\psi}(u,\omega)|^{2}du\right)\right\rbrace<\infty,~~\forall t>0
\end{align}
In this note we only consider (1.5), and a short new proof is given to establish that ${Z}(t)$ is a martingale. The exact criteria (1.5) and (1.6) are established, via a cluster-expansion method and utilising moment and cumulant-generating functions. 

Geometric Brownian motion arises naturally in problems of stochastic exponential growth, which occurs when a quantity of interest grows exponentially but with a random multiplier and/or random waiting time between spurts of growth. Such processes are ubiquitous in both Nature and in human-created systems: bacterial growth in a sustaining medium; in nuclear and cellular fission; tissue growth in embryonic and cancer biology; in viral epidemics; in financial markets and bubbles, and the Black-Scholes option pricing model; internet growth and propagation of viral social media posts; population and extinction dynamics; Moore's law for computer processing power; and inflationary expansion of the very early Universe with random fluctuations \textbf{[8-17]}. Noise can either boost stochastic exponential growth or drive a population to extinction \textbf{[8,11,12]}.

Most such models on random growth have focussed on the SDE for geometric Brownian motion since it can be solved exactly with a strong solution. Suppose a system evolves exponentially via the simple linear ODE $dX(t)=\alpha X(t)dt $ then $X(t)=X(0)\exp(\alpha t)$. If $\alpha>0$ then the system undergoes exponential growth and $X(t)\rightarrow\infty$ as $t\rightarrow \infty$ and if $\alpha<0$ then the system exponentially decays or collapses so that $X(t)\rightarrow 0$ as $t\rightarrow \infty$. The system is stable or static for $\alpha=0$. If the system is randomly perturbed by white noise (additively) then
\begin{align}
d{X}(t)=\alpha{X}(t)dt+\psi(t){X}(t)\xi(t)dt\equiv \alpha{X}(t)dt+\psi(t){X}(t)d{B}(t)
\end{align}
A strong solution exists for this SDE \textbf{[1]}. Given any $C^{2}$-functional $f({X}(t))$, the Ito Lemma is
\begin{align}
&df({X}(t))=\nabla_{X} f({X}(t))dX(t)+\frac{1}{2}\nabla^{2}_{X}f({X}(t))d[{X},{X}](t)\nonumber\\&
=\nabla_{X} f({X}(t))d{X}(t)+\frac{1}{2}\nabla^{2}_{X}f({X}(t))|\psi(t)|^{2}dt\nonumber\\&
=\nabla_{X} f({X}(t))\big\lbrace\alpha{X}(t)dt+\psi(t){X}(t)d{B}(t)\big\rbrace+\frac{1}{2}\nabla^{2}_{X}f({X}(t))|\psi(t)|^{2}dt
\end{align}
where $\nabla_{X}=d/d{X}(t)$ and $\nabla^{2}_{X}=d^{2}/d{X}(t)^{2}$. Using $f({X}(t))=\log {X}(t)$
\begin{align}
&d{U}(t)\equiv d\log {X}(t)=\frac{1}{{X}(t)}d{X}(t)+\frac{1}{2}\left(-\frac{1}{|{X}(t)|^{2}}\right)|\psi(t)|^{2}|{X}(t)|^{2}dt\nonumber\\&
=\frac{1}{{X}(t)}\big\lbrace\alpha {X}(t)dt+\psi(t){X}(t)d{B}(t)\big\rbrace-\frac{1}{2}\psi(t)|^{2}dt\nonumber\\&
=\left(\alpha-\frac{1}{2}|\psi(t)|^{2}\right)dt+\psi(t)d{B}(t)
\end{align}
Then the solution is
\begin{align}
X(t)=X(0)\exp(\alpha t)\exp\left(\int_{0}^{t}\psi(u)d{B}(u)-\frac{1}{2}\int_{0}^{t}|\psi(u)|^{2}du\right)=\exp(\alpha t){Z}(t)
\end{align}
The expectation is
\begin{align}
\bm{\mathsf{E}}{X}(t)=X(0)\exp(\alpha t)\bm{\mathsf{E}}\exp\left(\int_{0}^{t}\psi(u)d{B}(u)-\frac{1}{2}\int_{0}^{t}|\psi(u)|^{2}du\right)=
\exp(\alpha t)\bm{\mathsf{E}}{Z}(t)
\end{align}
which is $\bm{\mathsf{E}}{X}(t)=\exp(\alpha t)$ if $\bm{\mathsf{E}}Z(t)=1$. The martingale property of ${Z}(t)$ is therefore desirable and ensures there is no blow up at any finite time $t>0$.
\section{Main theorem and proof utilising a cluster expansion method}
\subsection{Preliminary definitions and lemmas}
Prior to the proof, we first establish the following preliminary definitions and lemmas. The proof utilises a cluster expansion method \textbf{[18-21]}.
\begin{defn}
White noise is (informally) defined as a zero-centred Gaussian process $\xi$ with covariance $\bm{\mathsf{E}}~\xi(t)\xi(s)=\delta(t-s)$ and 
expectation $\bm{\mathsf{E}}\xi(t)=0$. A scalar or inner product $\langle \bullet,\bullet \rangle$ in $L^{2}(\mathbb{R})$ can be defined as
\begin{align}
\langle\psi,\phi\rangle=\bm{\mathsf{E}}\big\langle \psi,\xi\big\rangle\big\langle \phi,\xi\big\rangle=\int\!\!\!\!\int\psi(s)\phi(t)\bm{\mathsf{E}}~\xi(s)\xi(t)dsdt=\int\!\!\!\!\int\psi(s)\phi(t)\delta(t-s)dsdt
\end{align}
Then
\begin{align}
&\langle\psi,\psi\rangle=\bm{\mathsf{E}}\big\langle \psi,\xi\big\rangle\big\langle \psi,\xi\big\rangle=\int\!\!\!\!\int\psi(s)\psi(t)\bm{\mathsf{E}}~\xi(s)\xi(t)dsdt=\int\!\!\!\!\int\psi(s)\psi(t)\delta(t-s)dsdt
\end{align}
which is
\begin{align}
\langle\psi,\psi\rangle=\int|\psi(s)|^{2}ds=\|\psi\|_{L^{2}(\mathbb{R})}^{2}
\end{align}
The noise $\xi$ can be formulated as a Gaussian random process on any space of distributions containing $\mathrm{L}^{2}(\mathbb{R})$. Stochastic 
integrals of the form $\int\psi(s)\xi(s)ds$ are well defined if $\psi\in\mathrm{L}^{2}(\mathbb{R})$. If $\psi$ is an indicator function on $[0,t]$ then
the process $B(t)=\int_{0}^{t}\xi(s)ds$ is a Brownian motion and $dB(t)=\xi(t)dt$.
\end{defn}
\begin{defn}
Given a time-ordered set $(t_{1},...t_{m})\in(0,T)$, with $t_{1}<t_{2}<t_{3}<...<t_{m-1}<t_{m}$, the mth-order moments and cumulants for the white noise $\xi(t)$ are given by
\begin{align}
&\overrightarrow{\mathbb{T}}\bm{\mathsf{E}}~\xi(t_{1})\otimes...\otimes\xi(t_{m})
=\overrightarrow{\mathbb{T}}\bm{\mathsf{E}}\prod_{q=1}^{m}\xi(t_{q})\nonumber\\&
\overrightarrow{\mathbb{T}}\bm{\mathsf{C}}~\xi(t_{1})\otimes...\otimes\xi(t_{m})
=\overrightarrow{\mathbb{T}}\bm{\mathsf{C}}\prod_{q=1}^{m}\xi(t_{q})
\end{align}
where $\mathbb{T}$ is a 'time-ordering operator'. For example if $t_{1}<t_{2}<t_{3}$ and $f(t)$ is a function of t then $\overrightarrow{\mathbb{T}}f(t_{2})f(t_{1})f(t_{3}=f(t_{1})f(t_{2})f(t_{3})$. Since $\xi(t)$ is a Gaussian, it is defined entirely by its first two moments and all cumulants or order $m\ge 3$ vanish so that
\begin{align}
\left\lbrace\overrightarrow{\mathbb{T}}\bm{\mathsf{C}}\prod_{q=1}^{m}\xi(t_{q})\right\rbrace_{m\ge 3}=0
\end{align}
\end{defn}
For example, at second order for a Gaussian process the binary cumulant is equivalent to the binary moment
\begin{align}
&\bm{\mathsf{C}}\xi(t_{1})\xi(t_{2})
=\bm{\mathsf{E}}~\xi(t_{1})\xi(t_{2})-\bm{\mathsf{E}}~\xi(t_{1})\bm{\mathsf{E}}~\xi(t_{2})=\bm{\mathsf{E}}\xi(t_{1})\xi(t_{2})=\delta(t_{2}-t_{1})
\end{align}
\begin{defn}
The moment generating function (MGF) $\mathscr{M}[\xi(t)]$ and the cumulant-generating function [CGF]
$\mathscr{C}[\xi(t)]$ of the white noise $\xi(t)$ are given by
\begin{align}
&{\mathscr{M}}[\xi(t)]=\sum_{m=0}^{\infty}\frac{\beta^{m}}{m!}\int_{0}^{t}...\int_{0}^{t_{m-1}}dt_{1}...dt_{m}\overrightarrow{\mathbb{T}}
\bm{\mathsf{E}}\prod_{q=1}^{m}~\xi(t_{q})\\&
{\mathscr{C}}[\xi(t)]=\sum_{m=1}^{\infty}\frac{\beta^{m}}{m!}\int_{0}^{t}...\int_{0}^{t_{m-1}}dt_{1}...dt_{m}
\overrightarrow{\mathbb{T}}\bm{\mathsf{C}}\prod_{q=1}^{m}\xi(t_{q})
\end{align}
where ${\beta}$ is an arbitrary real constant. It is important to note that the summation in (2.8) begins from $m=1$ and not $m=0$ and that
\begin{align}
&\overrightarrow{\mathbb{T}}\bm{\mathsf{E}}\prod_{q=1}^{m}\xi(t_{q})\ne\overrightarrow{\mathbb{T}}\prod_{q=1}^{m}\bm{\mathsf{E}}\xi(t_{q})\\&
\overrightarrow{\mathbb{T}}\bm{\mathsf{C}}\prod_{q=1}^{m}\xi(t_{q})\ne \overrightarrow{\mathbb{T}}\prod_{q=1}^{m}\bm{\mathsf{C}}\xi(t_{q})
\end{align}
Equations (2.7) and (2.8) can also be written as
\begin{align}
&{\mathscr{M}}[\xi(t)]=\sum_{m=0}^{\infty}\frac{\beta^{m}}{m!}
\int{\mathbf{D}}_{m}[t_{1}...t_{m}]\overrightarrow{\mathbb{T}}\bm{\mathsf{E}}\prod_{q=1}^{m}\xi(t_{q})\\&
{\mathscr{C}}[\xi(t)]=\sum_{m=1}^{\infty}\frac{\beta^{m}}{m!}
\int{\mathbf{D}}_{m}[t_{1}...t_{m}]\overrightarrow{\mathbb{T}}\bm{\mathsf{C}}\prod_{q=1}^{m}\xi(t_{q})
\end{align}
where $\int\bm{\mathrm{D}}_{m}[t]=\int...\int dt_{1}...dt_{m}$ is a 'path integral'. Choosing ${\beta}=+1$
\begin{align}
&{\mathscr{M}}[\xi(t)]=\sum_{m=0}^{\infty}\frac{1}{m!}\int\mathbf{D}_{m}[t_{1}...t_{m}]\overrightarrow{\mathbb{T}}\bm{\mathsf{E}}
\prod_{q=1}^{m}\xi(t_{q})\\&
{\mathscr{C}}[\xi(t)]=\sum_{m=1}^{\infty}\frac{1}{m!}\int\mathbf{D}_{m}[t_{1}...t_{m}]\overrightarrow{\mathbb{T}}\bm{\mathsf{C}}
\prod_{q=1}^{m}\xi(t_{q})
\end{align}
\end{defn}
\begin{lem}
The relation between the MGF and the CGF is
\begin{align}
{\log}{\mathscr{M}}[\xi(t)]={\mathscr{C}}[\xi(t)]
\end{align}
so that
\begin{align}
{\mathscr{M}}[\xi(t)]=\exp\big({\mathscr{C}}[\xi(t)]\big)
\end{align}
Hence
\begin{align}
&\sum_{m=0}^{\infty}\frac{1}{m!}\int\mathbf{D}_{m}[t_{1}...t_{m}]\overrightarrow{\mathbb{T}}
\bm{\mathsf{E}}\prod_{q=1}^{m}\xi(t_{q})\\&
\exp\left(\sum_{m=1}^{\infty}\frac{1}{m!}\int\mathbf{D}_{m}[t_{1}...t_{m}]\overrightarrow{\mathbb{T}}
\bm{\mathsf{C}}\prod_{q=1}^{m}\xi(t_{q})\right)
\end{align}
\end{lem}
\begin{prop}
Given $\psi\in\mathrm{L}^{2}(\mathbb{R})$ and the white noise $\xi(t)$, then for all $t>0$ define the Gaussian noise or random function
\begin{align}
\Xi(t)=\psi(t)\xi(t)
\end{align}
Then $\bm{\mathsf{E}}\Xi(t)=0$ and
\begin{align}
\bm{\mathsf{E}}~\Xi(t)\Xi(s)=\psi(t)\psi(s)\bm{\mathsf{E}}~\xi(t)\xi(s)=\psi(t)\psi(s)\delta(t-s)
\end{align}
\end{prop}
\begin{lem}
Given $\Xi(t)=\psi(t)\xi(t)$ then for any $t_{2}>t_{1}$ and $(t_{1},t_{2})\in (0,t)$
\begin{align}
\langle\psi,\psi\rangle\equiv\int_{0}^{t}\int_{0}^{t_{1}}\bm{\mathsf{E}}~\Xi(t_{1})\Xi(t_{2})dt_{1}dt_{2}
=\int_{0}^{t}|\psi(t_{1})|^{2}dt_{1}=\|\psi\|_{L_{2}(\mathbb{R})}
\end{align}
\end{lem}
\begin{proof}
Using the sifting property of the delta function
\begin{align}
&\int_{0}^{t}\int_{0}^{t_{1}}\bm{\mathsf{E}}~\Xi(t_{1})\Xi(t_{2}) dt_{1}dt_{2}\nonumber\\&
=\int_{0}^{t}\int_{0}^{t_{1}}\psi(t_{1})\psi(t_{2})\bm{\mathsf{E}}~\Xi(t_{1})\Xi(t_{2})dt_{1}dt_{2}\nonumber\\&
=\int_{0}^{t}\psi(t_{1})\bigg|\int_{0}^{t_{1}}\psi(t_{2})\delta(t_{2}-t_{1}) dt_{2}\bigg|dt_{1}\nonumber\\&=\int_{0}^{t}\psi(t_{1})\psi(t_{1})dt_{1}
=\int_{0}^{t}|\psi(t_{1})|^{2}dt_{1}\equiv\|\psi\|_{L_{2}(\mathbb{R^{+}})}
\end{align}
\end{proof}
\begin{prop}
Given $\Xi(t)=\beta(t)\xi(t)$ then the Dolean-Dades stochastic exponential ${Z}(t)$ can be expressed as
\begin{align}
&{Z}(t)=\exp\left(\int_{0}^{t}\psi(u)d{B}(u)-\frac{1}{2}\int_{0}^{t}|\psi(u)|^{2}du\right)\nonumber\\&
=\exp\left(\int_{0}^{t}\psi(u)\xi(u)du-\frac{1}{2}\int_{0}^{t}|\psi(u)|^{2}du\right)\nonumber\\&
=\exp\left(\int_{0}^{t}\Xi(u)du-\frac{1}{2}\int_{0}^{t}|\psi(u)|^{2}du\right)
\end{align}
\end{prop}
\subsection{Main theorem}
The main theorem and proof are now as follows:
\begin{thm}
The stochastic exponential ${Z}(t)$ is a true martingale iff
\begin{align}
\bm{\mathsf{E}}{Z}(t)\big\rbrace=\bm{\mathsf{E}}\exp\left(\int_{0}^{t}\psi(u)d{B}(u)-\frac{1}{2}\int_{0}^{t}|\psi(u)|^{2}du\right)=1
\end{align}
and requiring
\begin{align}
\exp(\|\psi\|_{L^{2}(\mathbb{R})}^{2})=\exp\left(\frac{1}{2}\int_{0}^{t}|\psi(u)|^{2}du\right)<\infty,~~\forall t>0
\end{align}
\end{thm}
\begin{proof}
\begin{align}
&\bm{\mathsf{E}}{Z}(t)=\bm{\mathsf{E}}\exp\left(\int_{0}^{t}\psi(u)d{B}(u)-\frac{1}{2}\int_{0}^{t}|\psi(u)|^{2}du\right)\nonumber\\&
=\bm{\mathsf{E}}\exp\left(\int_{0}^{t}\psi(u)\xi(u)du-\frac{1}{2}\int_{0}^{t}|\psi(u)|^{2}du\right)\nonumber\\&
=\bm{\mathsf{E}}\exp\left(\int_{0}^{t}\Xi(u)du-\frac{1}{2}\int_{0}^{t}|\psi(u)|^{2}du\right)\nonumber\\&
=\exp\left(-\frac{1}{2}\int_{0}^{t}|\psi(u)|^{2}du\right)\bm{\mathsf{E}}\exp\left(\int_{0}^{t}\Xi(u)du\right)\nonumber\\&
=\exp\left(-\frac{1}{2}\int_{0}^{t}|\psi(u)|^{2}du\right)\bm{\mathsf{E}}\sum_{m=0}^{\infty}\frac{1}{m!}\int_{0}^{t}dt_{1}...\int_{0}^{t_{m-1}}dt_{m}\prod_{q=1}^{m}
\Xi(t_{q})\nonumber\\&
\equiv \exp\left(-\frac{1}{2}\int_{0}^{t}|\psi(u)|^{2}du\right)\bm{\mathsf{E}}\sum_{m=0}^{\infty}\frac{1}{m!}\int \mathbf{D}_{m}[t_{1}...t_{m}]\prod_{q=1}^{m}
\Xi(t_{q})\nonumber\\&
=\exp\left(-\frac{1}{2}\int_{0}^{t}|\psi(u)|^{2}du\right)\sum_{m=0}^{\infty}\frac{1}{m!}\int \mathbf{D}_{m}[t_{1}...t_{m}]\bm{\mathsf{E}}\prod_{q=1}^{m}
\Xi(t_{q})
\end{align}
However, from (2.11), the MGF is
\begin{align}
{\mathscr{M}}[\Xi(t)]=\sum_{m=0}^{\infty}\frac{1}{m!}\int \mathbf{D}_{m}[t_{1}...t_{m}]\bm{\mathsf{E}}\prod_{q=1}^{m}\Xi(t_{q})
\end{align}
so (2.26) becomes
\begin{align}
\bm{\mathsf{E}}{Z}(t)=\exp\left(-\frac{1}{2}\int_{0}^{t}|\psi(u)|^{2}du\right){\mathscr{M}}[\Xi(t)]
\end{align}
Now using Lemma (2.4)
\begin{align}
{\mathscr{M}}[\Xi(t)]=\exp({\mathscr{C}}[\Xi(t)])
\end{align}
giving
\begin{align}
&\bm{\mathsf{E}}{Z}(t)=\exp\left(-\frac{1}{2}\int_{0}^{t}|\psi(u)|^{2}du\right){\mathscr{M}}[\Xi(t)]\nonumber\\&
=\exp\left(-\frac{1}{2}\int_{0}^{t}|\psi(u)|^{2}du\right)\exp\left({\mathscr{C}}[\Xi(t)]\right)\nonumber\\&
=\exp\left(-\frac{1}{2}\int_{0}^{t}|\psi(u)|^{2}du\right)\exp\left(\sum_{m=1}^{\infty}\frac{1}{m!}\int \mathbf{D}_{m}[t_{1}...t_{m}]\bm{\mathsf{C}}\overrightarrow{\mathbb{T}}\prod_{q=1}^{m}\Xi(t_{q})\right)\nonumber\\&
=\exp\left(-\frac{1}{2}\int_{0}^{t}|\psi(u)|^{2}du\right)\exp\left(\sum_{m=1}^{\infty}\frac{1}{m!}\int \mathbf{D}_{m}[t_{1}...t_{m}]\bm{\mathsf{C}}\overrightarrow{\mathbb{T}}
\prod_{q=1}^{m}\psi(t_{q})\xi(t_{q})\right)
\end{align}
Now since $\xi(t)$ is a Gaussian process, all cumulants of order $m\ge 3$ vanish so that
\begin{align}
\bm{\mathsf{C}}\overrightarrow{\mathbb{T}}\left\lbrace\prod_{q=1}^{m}
\Xi(t_{q})\right\rbrace_{m\ge 3}=\bm{\mathsf{C}}\overrightarrow{\mathbb{T}}\left\lbrace\prod_{q=1}^{m}
\psi(t_{q})\xi(t_{q})\right\rbrace_{m\ge 3}=0
\end{align}
This leaves
\begin{align}
&\bm{\mathbf{E}}{Z}(t)=\exp\left(-\frac{1}{2}\int_{0}^{t}|\psi(u)|^{2}du\right)\nonumber\\&\times \exp\left(\int \mathbf{D}_{m}[t_{1}]\psi(t_{1})
\bm{\mathsf{C}}\xi(t_{1})\big\rbrace+\frac{1}{2}\int\mathbf{D}[t_{1},t_{2}]\psi(t_{1})\psi(t_{2})\bm{\mathsf{C}}
\big\lbrace\xi(t_{1})\xi(t_{2})\right)\nonumber\\&
\equiv\exp\left(-\frac{1}{2}\int_{0}^{t}|\psi(u)|^{2}du\right)\nonumber\\&\times \exp\left(\int_{0}^{t}dt_{1}\psi(t_{1})
\bm{\mathsf{C}}\xi(t_{1})+\frac{1}{2}\int_{0}^{t}\int_{0}^{t_{1}}dt_{1}dt_{2}\psi(t_{1})\psi(t_{2})\bm{\mathsf{C}}
\xi(t_{1})\xi(t_{2})\right)\nonumber\\&\equiv \exp\left(-\frac{1}{2}\int_{0}^{t}|\psi(u)|^{2}du\right)\nonumber\\&\times \exp\left(\int_{0}^{t}dt_{1}\psi(t_{1})\underbrace{\bm{\mathsf{E}}~\lbrace\xi(t_{1})\big\rbrace}_{=0}+\frac{1}{2}\int_{0}^{t}\int_{0}^{t_{1}}dt_{1}dt_{2}\psi(t_{1})\psi(t_{2})
\bm{\mathsf{E}}~\xi(t_{1})\xi(t_{2})\right)
\nonumber\\&\equiv \exp\left(-\frac{1}{2}\int_{0}^{t}|\psi(u)|^{2}du\right)\exp\left(\frac{1}{2}\int_{0}^{t}\int_{0}^{t_{1}}dt_{1}dt_{2}\psi(t_{1})\psi(t_{2})\bm{\mathsf{E}}~
\xi(t_{1})\xi(t_{2})\right)\nonumber\\&\equiv \exp\left(-\frac{1}{2}\int_{0}^{t}|\psi(u)|^{2}du\right)\exp\left(\frac{1}{2}\int_{0}^{t}\int_{0}^{t_{1}}dt_{1}dt_{2}{\psi}(t_{1})
{\psi}(t_{2}){\delta}(t_{2}-t_{1})\right)\nonumber\\&
=\exp\left(-\frac{1}{2}\int_{0}^{t}|\psi(u)|^{2}du\right)\exp\left(\frac{1}{2}\int_{0}^{t}dt_{1}{\psi}(t_{1})
\bigg|\int_{0}^{t_{1}}dt_{2}{\psi}(t_{2}){\delta}(t_{2}-t_{1})\bigg|\right)\nonumber\\&
=\exp\left(-\frac{1}{2}\int_{0}^{t}|\psi(u)|^{2}du\right)\exp\left(\frac{1}{2}\int_{0}^{t}dt_{1}\psi(t_{1})\psi(t_{1})\right)\nonumber\\&
=\exp\left(-\frac{1}{2}\int_{0}^{t}|\psi(u)|^{2}du\right)\exp\left(\frac{1}{2}\int_{0}^{t}|\psi(t_{1})|^{2}dt_{1}\right)\nonumber\\&
\equiv\exp\left(-\frac{1}{2}\int_{0}^{t}|\psi(u)|^{2}du\right)\exp\left(\frac{1}{2}\int_{0}^{t}|\psi(u)|^{2}du\right)\nonumber\\&
\equiv \exp\left(-\frac{1}{2}\int_{0}^{t}|\psi(u)|^{2}du+\frac{1}{2}\int_{0}^{t}|\psi(u)|^{2}du\right)=1
\end{align}
iff
\begin{align}
\exp\left(\dfrac{1}{2}\int_{0}^{t}|\psi(u)|^{2}du\right)<\infty
\end{align}
which are the Novikov criteria. If for some function $\psi(t),~\exists~ T>0$ such that $\frac{1}{2}\int_{0}^{T}|\psi(u)|^{2}du=\infty$ then $\exp\left(\tfrac{1}{2}\int_{0}^{T}|\psi(u)|^{2}du\right)=\infty$ and $\exp\left(-\tfrac{1}{2}\int_{0}^{T}|\psi(u)|^{2}du\right)=0$ so that $\bm{\mathsf{E}}Z(t)=0$. Hence ${Z}(t)$ is a martingale for these criteria and the proof is complete.
\end{proof}
As a  corollary, it follows easily that $\bm{\mathsf{M}}(t,p)=\bm{\mathsf{E}}[|{Z}(t)|^{p}]$ is a submartingale for all $p>1$ if
$\phi(t)=\int_{0}^{t}|\psi(u)|^{2}du$ is monotone increasing with t.
\begin{cor}
Given $\bm{\mathsf{E}}{Z}(t)=1$ it follows that
\begin{align}
\bm{\mathsf{M}}(t,p)=\bm{\mathsf{E}}[|{Z}(t)|^{p}]=\exp\left(\tfrac{1}{2}[p(p-1)\phi(t)\right)
\end{align}
is a submartingale for all $p>1$, if $\phi(t))=\int_{0}^{t}|\psi(u)|^{2}du $ is bounded but monotone increasing with t.
\end{cor}
\begin{proof}
\begin{align}
&|{Z}(t)|^{p}=\left(\exp\left|\int_{0}^{t}\psi(u)d{B}(u)-\frac{1}{2}\int_{0}^{t}|\psi(u)|^{2}du\right)\right|^{p}\nonumber\\&
=\exp\left(p\int_{0}^{t}\psi(u)d {B}(u)-\frac{p}{2}\int_{0}^{t}|\psi(u)|^{2}du\right)\nonumber\\&
=\exp\left(-\frac{p}{2}\int_{0}^{t}|\psi(u)|^{2}du\right)\exp\left(p\int_{0}^{t}\psi(u)d{B}(u)\right)
\end{align}
Then
\begin{align}
\bm{\mathsf{M}}(t,p)=&\bm{\mathsf{E}}[|{Z}(t)|^{p}]=\exp\left(-\frac{p}{2}\int_{0}^{t}|\psi(u)|^{2}du\right)
\bm{\mathsf{E}}\exp\left(p\int_{0}^{t}\psi(u)d{B}(u)\right)\nonumber\\&
=\exp\left(-\frac{1}{2}p\int_{0}^{t}|\psi(u)|^{2}du\right)\exp\left(\frac{1}{2}p^{2}\int_{0}^{t}|\psi(u)|^{2}du\right)\nonumber\\&
=\exp\left(\frac{1}{2}p(p-1)\int_{0}^{t}|\psi(u)|^{2}du\right)=\exp\left(\tfrac{1}{2}p(p-1)\phi(t)\right)
\end{align}
If $\phi(t)>\phi(t^{\prime})$ for all $t>t^{\prime}$ then $\bm{\mathsf{M}}(t,p)>\bm{\mathsf{M}}(t^{\prime},p)$ so that $\bm{\mathsf{M}}(t,p)$ is monotone increasing.
Hence, $\bm{\mathsf{M}}(t,p)$ is a submartingale on $\mathbb{R}^{+}$. If $p=1$ then $\bm{\mathsf{M}}(t,1)=\bm{\mathsf{E}}{Z}(t)=1$.
\end{proof}

}
\end{document}